\documentclass[10pt]{amsart}
\usepackage{latexsym,amsmath,amsfonts,amscd,amssymb}
\usepackage{graphics}
\newtheorem{theorem}{Theorem}[section]
\newtheorem{lemma}[theorem]{Lemma}
\newtheorem{proposition}[theorem]{Proposition}
\newtheorem{corollary}[theorem]{Corollary}

\newtheorem{definition}[theorem]{Definition}
.

\theoremstyle{remark}

\newcommand{\Div}{\operatorname{Div}}

\newcommand{\supp}{\operatorname{supp}}

\newcommand{\Li}{\operatorname{Li}}

\newcommand{\cC}{{\mathcal C}}

\newcommand{\cF}{{\mathcal F}}

\newcommand{\cM}{{\mathcal M}}

\newcommand{\cO}{{\mathcal O}}
\newcommand{\cP}{{\mathcal P}}

\newcommand{\cT}{{\mathcal T}}

\newcommand{\CC}{{\mathbb C}}

\newcommand{\PP}{{\mathbb P}}
\newcommand{\QQ}{{\mathbb Q}}

\newcommand{\ZZ}{{\mathbb Z}}

\renewcommand{\a}{\alpha}

\newcommand{\eps}{\epsilon}

\begin{document}
\title[E\~ne product in the  transalgebraic class]{E\~ne product in the  transalgebraic class}


\author[R. P\'{e}rez-Marco]{Ricardo P\'{e}rez-Marco}

\address{Ricardo P\'erez-Marco\newline 
\indent  Institut de Math\'ematiques de Jussieu-Paris Rive Gauche, \newline
\indent CNRS, UMR 7586, \newline
\indent  Sorbonne Universit\'e, B\^at. Sophie Germain, \newline 
\indent 75205 Paris, France}

\email{ricardo.perez-marco@imj-prg.fr}

\maketitle

{ \centerline{\sc Abstract}}

% \bigskip

\begin{minipage}{14cm}
\noindent
We define transalgebraic functions on a compact Riemann surface as meromorphic functions 
except at a finite number of punctures where they have finite order exponential singularities. This transalgebraic class 
is a topological multiplicative group. We extend the action of the  e\~ne product to the transcendental class on the Riemann sphere. 
This transalgebraic class, modulo constant functions, is a commutative ring for the multiplication, as the additive structure, and 
the e\~ne product, as the multiplicative structure. In particular, 
the divisor action of the e\~ne product by multiplicative convolution extends to these 
transalgebraic divisors. The polylogarithm hierarchy appears related 
to transalgebraic e\~ne poles of higher order.
\end{minipage}

%\newpage

\tableofcontents

\bigskip
\bigskip
\bigskip

\begin{figure}[h]
\centering
\resizebox{12cm}{!}{\includegraphics{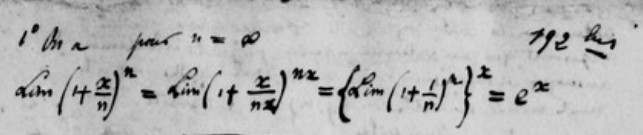}}    % name of the file - without extension
\caption{From \'Evariste Galois' manuscripts, Ms 2108, Institut de France, \cite{Ga}.} 
\end{figure}

\newpage

%%%%%%%%%%%%%%%%%%%%%%%%%%%%%%%%%%%%%%%%%%%%%%%%%%%%%%%%%%%%%%%%%%
\section{Introduction and Euler miscellanea} \label{sec:eule}
%%%%%%%%%%%%%%%%%%%%%%%%%%%%%%%%%%%%%%%%%%%%%%%%%%%%%%%%%%%%%%%%%%

One of Euler's major discoveries is the transalgebraic nature of the
exponential function as the unique ``polynomial'', normalized to take the value $1$ at $0$ as well as 
its derivative, with only one zero of infinite
order at $\infty$ (see  \cite{Eu}). Euler writes
$$
e^{-z} =\left ( 1-\frac{z}{\infty} \right )^{+\infty} \ .
$$

Hence, the exponential appears as a ``transalgebraic polynomial''.
Observe that in this heuristic formula both infinite symbols are of a different nature: One is an infinite 
point ($\infty$ in the Riemann sphere) and the other an infinite number ($+\infty$ is the infinite order of 
the zero). Obviously, the proper justification of this heuristic comes from  Euler formula,
$$
e^z=\lim_{n\to +\infty} \left(1+\frac{z}{n}\right )^n 
$$
A one line proof, assuming known the existence of the elementary limit for $e=2.71828182846\ldots $
$$
e=\lim_{n\to +\infty} \left(1+\frac{1}{n}\right )^n 
$$
can be found in Galois manuscripts \cite{Ga} (these may be course notes from 
his professor P.-L. \'E. Richard at Lyc\'ee Louis-le-Grand,  so the authorship is unclear),
$$
\lim_{n\to +\infty} \left(1+\frac{z}{n}\right )^n = \lim_{n\to +\infty} \left(1+\frac{z}{nz}\right )^{nz}
=\lim_{n\to +\infty} \left [\left(1+\frac{1}{n}\right )^n\right ]^z =e^z   
$$
Note, that we do not want to use the logarithm function in the proof since the proper order to develop the theory is to define the exponential first.
We also have a purely geometric proof that is the proper geometrization of Euler's transalgebraic heuristics, by using 
Carath\'eodory's convergence Theorem of the uniformizations for the Carath\'eodory convergence of the log-Riemann 
surfaces of $\sqrt[n]{z}$ to the log-Riemann surface of $\log z$ when $n\to +\infty$ (see \cite{BPM1} for details, and 
\cite{BPM2} for background on log-Riemann surfaces). The construction of the log-Riemann
surface of $\log z$ does not require the previous definition of the logarithm (on the contrary, we can define the logarithm function from it).

\medskip

The main property of the exponential can be derived ``\`a la Euler'' as follows 
(using $\infty^{-2} << \infty^{-1}$):

\begin{align*}
 e^{z_1}.e^{z_2} &= \left ( 1+\frac{z_1}{\infty} \right )^{+\infty} . \left ( 1+ \frac{z_2}{\infty} \right )^{+\infty} \\
&=\prod_{+\infty} \left ( 1+\frac{z_1}{\infty} \right ). \left ( 1+\frac{z_2}{\infty} \right ) \\
&=\prod_{+\infty} \left ( 1+\frac{z_1 +z_2}{\infty} + \frac{z_1 .z_2}{\infty^2} \right ) \\
&=\prod_{+\infty} \left ( 1+\frac{z_1 +z_2}{\infty} \right ) \\
&=\left ( 1+\frac{z_1 +z_2}{\infty} \right )^{+\infty} \\
&=e^{z_1+z_2} 
\end{align*}

The exponential function is the link between the additive and the multiplicative 
structure on $\CC$. It serves also as the link between the multiplicative and the e\~ne product 
structure. 

\medskip

\textbf{The e\~ne product.}

\medskip

We briefly recall the definition of the e\~ne product (see \cite{PM1}). 
Given two polynomials $P,Q\in \CC[z]$, normalized such that 
$$
P(0)=Q(0)=1
$$
say
\begin{align*}
P(z) &=1+a_1 z+a_2 z^2+\ldots =\prod_{\alpha} \left (1-\frac{z}{\alpha}\right )\\ 
Q(z) &=1+b_1 z+b_2 z^2+\ldots =\prod_{\beta} \left (1-\frac{z}{\beta}\right )
\end{align*}
where $(\alpha)$ and $(\beta)$ are the respective zeros counted with multiplicity,
then we define the e\~ne product (\cite{PM1}) by
$$
P\star Q(z)=\prod_{\alpha, \beta} \left (1-\frac{z}{\alpha \beta}\right )
$$
Therefore, the divisor of $P\star Q$ is the multiplicative convolution of the divisors of $P$ 
and $Q$. If we write,
$$
P\star Q(z) =1+c_1 z+c_2 z^2+\ldots 
$$
then for $n\geq 1$,
$$
c_n=C_n(a_1, a_2,\ldots, a_n, b_1, b_2, \ldots, b_n)=-n a_n b_n +\ldots
$$
where $C_n\in \ZZ[X_1, \ldots, X_n, Y_1,\ldots, Y_n]$ is a universal polynomial with integer 
coefficients and the dots in the left side of the formula represents a polynomial on $a_1, \ldots 
,a_{n-1}, b_1,\ldots ,b_{n-1}$ (see \cite{PM1}). This allows to define the e\~ne product $f\star g$ of two 
formal power series 
\begin{align*}
f(z) &=1+a_1 z+a_2 z^2+\ldots \\ 
g(z) &=1+b_1 z+b_2 z^2+\ldots 
\end{align*}
with coefficients $a_n, b_n\in A$, in a general commutative ring $A$ by 
$$
f\star g(z) =1+c_1 z+c_2 z^2+\ldots 
$$
with $c_n=C_n(a_1, a_2,\ldots, a_n, b_1, b_2, \ldots, b_n)$ (so without any 
reference to the zeros that we don't have in this general setting).

We refer to \cite{PM1} for the rich algebraic and analytic properties of the 
e\~ne product. From the previous Eulerian heuristic of the exponential of a polynomial with an infinite order zero at $\infty$, it is 
natural to expect that the e\~ne product with an exponential 
must make sense and be an exponential. We do have a much more precise result: The 
exponential linearizes the e\~ne product. More precisely, we have a linear  
the exponential form of the e\~ne product (Theorem 4.1 from \cite{PM1}) in the 
following sense: If we write
the power series in exponential form
\begin{align*}
f(z) &=e^{F(z)} \\ 
g(z) &=e^{G(z)} 
\end{align*}
with formal power series (that have a finite non-zero radius of convergence when $f$ and $g$ 
are polynomials),
\begin{align*}
F(z) &=F_1 z+F_2 z^2+\ldots \\ 
G(z) &=G_1 z+G_2 z^2+\ldots 
\end{align*}
then we have
$$
f\star g(z) = e^{H(z)}
$$
with
$$
H(z)=-F_1 G_1 z -2 F_2 G_2 z^2 -3 F_3 G_3 z^3 +\ldots=-\sum_{k=0}^{+\infty} kF_k G_k \, z^k
$$
We denote $\star_e$ the linearized exponential form of the e\~ne product
$$
F\star_e G (z)=-\sum_{k=0}^{+\infty} kF_k G_k \, z^k
$$
Note that we could have defined the e\~ne product for polynomials in this formal way, but we would miss  
the original interpretation  with the convolution of zeros (in particular because
the disk of convergence of the exponential form never contains zeros!).

\medskip

With similar heuristic ideas as before we can derive \`a la Euler the exponential form of the
e\~ne product (for a rigorous proof see Theorem 4.1 from \cite{PM1}).
Write as before
\begin{align*}
f(z) &= e^{F(z)}=\left (1+\frac{F(z)}{\infty}\right )^{+\infty} \\
g(z) &= e^{G(z)}=\left (1+\frac{G(z)}{\infty }\right )^{+\infty}
\end{align*}
Then, note the double $\infty$ on the products and $\infty^{-3}<<\infty^{-2}$, and compute

\begin{align*}
f\star g(z) &= \left (1+\frac{F(z)}{\infty}\right )^{+\infty} \star
\left (1+\frac{G(z)}{\infty }\right )^{+\infty}  \\
&= \prod_{+\infty , +\infty } \left (1+\sum_{k\geq 1} \frac{F_k}{\infty}
z^k \right ) \star \left (1+\sum_{k\geq 1} \frac{G_k}{\infty}
z^k \right ) \\
&= \prod_{+\infty , +\infty } \left ( 1+ \sum_k \left ( -k \frac{F_k}{\infty} .
\frac{G_k}{\infty} + \cO \left (\frac{1}{\infty^3} \right )\right ) z^k \right ) \\
&=\prod_{+\infty , +\infty} \left ( 1+ \frac{1}{\infty . \infty}
\sum_{k\geq 1}  -k  F_k G_k \, z^k \right ) \\
&=\left ( 1+\frac{(F\star_e G) (z)}{\infty .\infty} \right )
^{(+\infty) .(+\infty)} \\
&= e^{(F \star_e G)(z)} 
\end{align*}

Motivated by these transalgebraic heuristic considerations, the  purpose of this 
article is to extend the e\~ne product to the class of transalgebraic functions, a class of 
functions with exponential singularities that we define precisely in next section. 
This is a good demonstration of the dual analytic-algebraic 
character of the e\~ne product.

\pagebreak

%%%%%%%%%%%%%%%%%%%%%%%%%%%%%%%%%%%%%%%%%%%%%%%%%%%%%%%%%%%%%%%%%%
\section{The transalgebraic class on a compact Riemann surface.} \label{sec:comp}
%%%%%%%%%%%%%%%%%%%%%%%%%%%%%%%%%%%%%%%%%%%%%%%%%%%%%%%%%%%%%%%%%%

We first define exponential singularities. The exponential singularities of finite order play the role of zeros of infinite order
following Euler's heuristics.
We are mostly interested in this section in the case of the Riemann sphere, or $1$-dimensional projective space over $\CC$, 
$X=\overline{\CC}=\PP^1(\CC)$ (genus $g=0$), but there is little extra effort to 
define the transalgebraic class of functions $\cT(X)$ for a general compact Riemann surface $X$. 
We denote $\cM(X)$ the space of meromorphic functions on $X$, and $\cM(X)^*$ the non-zero meromorphic functions.

\begin{definition}
A point $z_0\in  \CC$ is an exponential singularity of $f$ if for some neighborhood $U$ of $z_0$, 
$f$ is a holomorphic function $f:U-\{z_0\}\to \CC$, $f$ has no zeros nor poles on $U$ and $f$ does not extend 
meromorphically to $U$. 

The exponential singularity $z_0\in \CC$ of $f$ is of finite order $1\leq d=d(f,z_0) < +\infty$ if $d$ is the minimal integer  such that 
$$
\limsup_{z\to z_0} |z-z_0|^d \log |f(z)| <+\infty
$$
If no such finite order $d$ exists, the exponential singularity is of infinite order and $d=d(f,z_0)=+\infty$.

Let $X$ be a Riemann surface. A point $z_0\in X$ is an exponential singularity for $f$ if it is an exponential 
singularity in a local chart. The order 
$d\geq 1$ has the same definition and is independent of the local chart. 
\end{definition}

Observe that we cannot have $d=0$ because $f$ would be bounded in a pointed neighborhood of $z_0$ and by Riemann's removability Theorem $f$ 
will have an holomorphic extension at $z_0$.
Note also that the definition means that $f$ has no monodromy around $z_0$ (i.e. $f$ is holomorphic in a pointed neighborhood of $z_0$),  
and that $z_0$ is not a regular point nor a pole for $f$.
Also  $z_0$ is an exponential singularity for $f$ if and only if it is also one for $f^{-1}$. We have a more precise 
result.

\begin{proposition}
The point $z_0\in X$ is an exponential singularity for $f$ if and only if there is a local chart $z$ where 
$z_0=0$ and $f$ can be written in a neighborhood $U$ of $0$  as
$$
f(z)=z^ne^{h(z)}
$$
where $n\in \ZZ$ and $h:U-\{0\}\to \CC$ is holomorphic. The integer $n\in \ZZ$ if the residue of the logarithmic differential $d\log f=f'/f\, dz$ at $z=0$.
The order $d$ is finite if and only if  $h$ is meromorphic with a pole of order $d$ at $z_0$. The order of 
the pole of the logarithmic differential of $f$ at $z_0$ is $d+1\geq 2$.
\end{proposition}

\begin{proof}
The problem is local and we  can assume $z_0=0$ and take $U$ a simply connected neighborhood of $0$ in a local chart with local variable $z$ 
where $f$ has no zeros nor poles on $U$. Then, the logarithmic derivative $f'/f$ is 
holomorphic on $U-\{0\}$ and has a Laurent expansion
$$
\frac{f'(z)}{f(z)} =\sum_{n\in \ZZ} a_n z^n
$$
The residue $a_{-1}$ is an integer $n\in \ZZ$ since $f$ has no monodromy around $0$. We take for $h$ the holomorphic function on $U-\{z_0\}$
defined by
$$
h(z)=c_0+\sum_{n\in \ZZ^*} \frac{a_{n-1}}{n} z^{n}
$$
where $c_0\in \CC$ is an arbitrary constant to be chosen later. Then we have
$$
\frac{f'(z)}{f(z)} =\frac{n}{z}+h'(z)
$$
and, choosing the constant $c_0$ properly we have
$$
f(z)=z^ne^{h(z)} 
$$
as desired. Conversely, such an expression has clearly an exponential singularity at $z_0=0$.
From this expression, the order $d$ is finite if and only if $h$ has finite polar part of order $d$ which is equivalent to have a pole of order $d+1$ for 
the logarithmic derivative of $f$.
\end{proof}

\begin{proposition}
If $f$ has an exponential singularity at $z_0$ then in any pointed neighborhood of $z_0$ the function $f$ takes any value $c\in \CC^*$ 
infinitely often.
\end{proposition}

\begin{proof}
This is a direct application of Picard's Theorem since $f$ does not take the values $0, \infty \in \overline{\CC}$ in a small pointed neighborhood of $z_0$. 
\end{proof}

\begin{definition}\label{def:punctured}
For a compact Riemann surface $X$ and a finite number of punctures $S\subset X$, we define the space $\cF (X,S)$ as the set of 
non-zero meromorphic functions $f$ on $X-S$ such that  $S$ is a set of zeros, poles, or exponential 
essential singularities of $f$ or of its meromorphic extension. We define also
$$
\cF (X)=\bigcup_{S\subset X; S \, \textrm{finite}} \cF (X,S)
$$
\end{definition}

\textbf{Examples.}

\medskip
$\bullet$ For $X=\overline{\CC}$ and $S=\emptyset$ we have that $\cF(\overline{\CC}, \emptyset)=\CC(z)^*$ is the space of non-zero rational functions. 

\smallskip
$\bullet$ For $X=\overline{\CC}$ and $S=\{\infty\}$, $\cF(\overline{\CC}, \{\infty\})$ is the set of functions $f(z)=R(z) e^{h(z)}$ where 
$R\in \CC(z)^*$ and $h:\CC\to \CC$ is an entire function, in the chart $z\in \CC =\overline{\CC}-\{\infty\}$. 
The exponential singularity at infinite is of finite order $d\geq 1$ if and only if $h\in \CC[z]$ is a polynomial of degree $d\geq 1$.
For $d=0$, $h$ must be constant and $f$ is a non-zero rational function.

\medskip

\begin{proposition}
The spaces $\cF (X,S)$ and $\cF (X)$ endowed with the multiplication of functions are multiplicative abelian groups.
If $f\in \cF(X)$ then $f$ has a finite number of zeros, poles and exponential singularities. 

\end{proposition}

\begin{proof}
These spaces are multiplicative groups from the remark made previously that $z_0\in X$ is an exponential singularity for $f$ if and only if it is one for $f^{-1}$.
The finiteness of exponential singularities follows from the finiteness of $S$.
By compactness, an infinite sequence of zeros must have an accumulation point on $X$. It cannot accumulate a point of $X-S$ or $f$ would be identically zero. It can neither
accumulate a point of $S$ since they all have small zero free pointed neighborhoods. Hence we must have a finite number of zeros. The same argument for poles 
(or applied to $f^{-1}$) gives a finite number of them. 
\end{proof}

\begin{definition}
Let $X$ be a compact Riemann surface, $S\subset X$ a finite subset, and $n\geq 0$. 
We define $\cF_n(X)\subset \cF(X)$ (resp. $\cF_n(X, S)\subset \cF(X, S)$) as the subset of functions having at most $n$ zeros and poles 
(both \underline{not} counted with multiplicity), and exponential singularities (resp. with the exponential singularities located at $S$). We define 
$\cT_n(X)\subset \cF_n(X)\subset \cF(X)$, resp. $\cT_n(X, S)\subset \cF_n(X,S)\subset \cF(X,S)$,  as 
the subset of functions with finite order exponential singularities.

We have the filtrations
\begin{align*}
\cF(X) &= \bigcup_{n\geq 0} \cF_n(X) \\
\cF(X, S) &= \bigcup_{n\geq 0} \cF_n(X,S) 
\end{align*}
We define the class of transalgebraic functions $\cT(X)$ and $\cT(X,S)$ by
\begin{align*}
\cT(X) &= \bigcup_{n\geq 0} \cT_n(X) \\
\cT(X,S) &= \bigcup_{n\geq 0} \cT_n(X,S) 
\end{align*}
We define also $\cM_n(X) \subset \cF_n(X)$ to be the subset of meromorphic functions. We have, for $n\geq 0$,
\begin{align*}
\cM_n(X) &\subset \cT_n(X) \subset \cF_n(X) \\
\cM_n(X) &\subset \cT_n(X,S) \subset \cF_n(X,S) \\
\end{align*}
and 
\begin{align*}
\cM(X) &\subset \cT(X) \subset \cF(X)  \\
\cM(X) &\subset \cT(X,S) \subset \cF(X,S)   
\end{align*}
\end{definition}

\textbf{Remarks.}

$\bullet$ The space $\cF_0(X)=\CC^*$ is the set of non-zero constant functions. 
\medskip

$\bullet$ For $X=\overline{\CC}$ and $n=1$, $\cF_1(\overline{\CC})$ is the set of functions 
which are Moebius conjugated to some $e^{h(z)}$ where $h$ is an entire function. To see this, observe first that  
the function cannot be a constant nor a rational non constant function since for such a function
the number of zeros and poles would be at least $2$. For $f\in \cF_1(\overline{\CC})$ We can send the unique singularity to $\infty$ by a Moebius map, 
and $f$ would be of the form $e^{h(z)}$ with $h$ holomorphic in $\CC$. It follows that $\cT_1(\overline{\CC})$ are 
those functions Moebius conjugated to $e^{h(z)}$ where $h(z)\in \CC[z]$ is a polynomial.

\medskip

\begin{proposition}
The spaces $\cT (X)$ and $\cT(X,S)$ endowed with the multiplication are  groups. 
\end{proposition}

\begin{proof}
From previous remarks they are invariant by taking inverses and their are clearly multiplicative invariant. 
\end{proof}

We define transalgebraic divisors.

\begin{definition}
For $f\in \cF(X)$, $S(f)$ denotes the set of exponential singularities of $f$. We define the transalgebraic divisor of $f$ as 
the formal sum
$$
\Div(f) = \sum_{\rho \in X} n_\rho . (\rho ) +\sum_{\rho \in S(f)} d_\rho .(\rho )_\infty
$$
where $n_\rho \in \ZZ$ is the positive, resp. negative, order of the zero, resp. pole, at $\rho$ 
or $n_\rho$ is the residue of the logarithmic derivative of $d\log f$ at $\rho$, or $n_\rho =0$ if $\rho$ is neither a zero nor pole nor 
an exponential singularity, and $1\leq d_\rho\leq +\infty$ is the order of the 
exponential singularity at $\rho$ when $\rho$ is an exponential singularity, i.e. 
$d_\rho +1$ is the  order of the polar part of $d\log f$ at $\rho$. The integer $d_\rho=d_\rho(f)$ is also called the transalgebraic degree of $f$ 
at $\rho$.

The algebraic part of the divisor is
$$
\Div_0(f) = \sum_{\rho \in X} n_\rho . (\rho )
$$
and the transcendental part of the divisor is
$$
\Div_\infty(f) = \sum_{\rho \in S(f)} d_\rho .(\rho )_\infty
$$
so that 
$$
\Div(f) = \Div_0(f) + \Div_\infty(f) \ .
$$
The support of the transcendental part of the divisor is $\supp(\Div_\infty(f))=S(f)$. The 
support of the algebraic part $\supp(\Div_0(f))$ is the classical support of $\Div_0(f)$.

The support of the transalgebraic divisor of $f$ is the finite subset of $X$
$$
\supp (\Div (f)) = \supp(\Div_0(f)) \cup \supp(\Div_\infty(f)) \ .
$$
\end{definition}

With these notations, we have $f\in \cF(X, S(f))\cap \cF_{|\supp(f)|}$.

We recall that the set of compact subsets of a compact metric space is endowed with a natural Hausdorff distance and 
a Hausdorff topology. Any distance on $X$ defining its compact surface topology defines the same Haussdorf topology 
on compact subsets. We fix one such distance $d_H$ and we define a topology on $\cF(X)$ that is independent of the 
choice of $d_H$.

\begin{definition}
We define the topology on $\cF(X)$ of uniform
convergence out of the support of the divisor  by defining a
sequence $(f_k)$ convergent to $f \in \cF$ if
$$
\supp (\Div (f_k)) \to \supp (\Div (f))
$$
in Hausdorff topology, and $f_k \to f$ uniformly on compact sets out of
${\hbox {\rm supp}} (\Div (f))$.
\end{definition}

We can construct a bases of neighborhoods $(U_\epsilon(f))_{\epsilon >0}$ 
of an element $f\in \cF(X)$  for this 
topology by taking for $\epsilon >0$, the $\epsilon$-Hausdorff neighborhood 
of $\supp(\Div (f))$, $V_\eps(\supp(\Div (f)))$ and defining $U_\epsilon(f)$ 
to be the subset of $g\in \cF(X)$ 
such that $\supp(\Div (g))\in V_\eps(\supp(\Div (f)))$, i.e.
$$
d_H(\supp(\Div (g)), \supp (\Div (f))) <\epsilon
$$
and
$$
||g-f||_{C^0(X-W_\eps(\supp(\Div (f)))}<\eps
$$

\noindent where $W_\eps(\supp(\Div (f)))$ denotes the $\eps$-neighborhood of $\supp(\Div (f))$ in $X$.

\begin{proposition}
 The groups  $\cM(X)^*$, $\cM(X,S)^*$,  $\cT(X)$, $\cT(X,S)$, $\cF(X)$ and $\cF(X,S)$  are topological groups.
\end{proposition}

\begin{proof}
For the groups of fuctions with singularities in $S$, the multiplication and inverse are continuous since the set of zeros, 
poles and singularities are restricted to $S$. The larger groups are unions of those according to 
the filtration from Definition \ref{def:punctured}).
\end{proof}

\begin{proposition}
The subgroup $\cF(X,S)$ and the subspace $\cF_n(X)$ are closed in $\cF(X)$.
\end{proposition}

\begin{proof}
We prove first that $\cF_n(X)$ is closed in $\cF(X)$.
 Consider a sequence $(f_k)\subset \cF_n(X)$ such that $f_k\to f \in \cF(X)$. The cardinal of finite sets 
 is upper semi continuous for the Hausdorff topology, hence $|\supp(\Div (f))| \leq n$. Moreover,
by Hurwitz Theorem the limit function $f$ has no singularities and 
cannot have zeros nor poles in $X-\supp (\Div(f))$ (note that $f$ cannot be the constant 
function $0$ or $\infty$ since $0, \infty \notin \cF$). Hence $f\in \cF_n(X)$. Now the 
subspace $\cF_n(X, S)$ is closed in $\cF(X)$ with the same proof.
\end{proof}

The space $\cM_n(X)^*$ is not closed as we can see using the Euler example where $X=\overline{\CC}$ with 
$$
f_k(z)=\left (1+\frac{z}{k}\right )^k \in \cM_2(\overline{\CC})^*
$$
but $f_k(z)\to e^z \notin \CC(z)= \cM(\overline{\CC})$.

We can define a transalgebraic degree:

\begin{definition}
Let $f\in \cT(X)$. 
The total transalgebraic degree of $f$ is 
$$
d_\infty(f) =\sum_{\rho \in S(f)} (d_\rho (f)+1)
$$
We also define $d_0(f) = |\supp(\Div_0(f))-S(f)|$.
We define the space $\cT^{d_0, d_\infty}(X) \subset \cT (X)$ as the subspace of those $f\in \cT(X)$ with
\begin{align*}
d_\infty(f) &= d_\infty \\
d_0(f) &= d_0 \ .
\end{align*}
For a finite set $S\subset X$ we define 
$$
\cT^{d_0, d_\infty}(X, S)=\cT^{d_0, d_\infty}(X) \cap \cT(X, S)
$$
\end{definition}

\noindent Observe that for $d_\infty =0 $,
$$
\cT^{d_0, 0}(X) =  \cM_{d_0} (X)
$$
and for $d_\infty  \geq 1 $,
$$
\cT^{d_0, d_\infty}(X) \subset  \cT_{d_0+d_\infty-1} (X)
$$
since for any $f \in \cT^{d_0, d_\infty}(X)$ we have $|\supp(\Div(f))| \leq d_0(f)+d_\infty(f) -1$.
Also we have  $|S(f)|\leq d_\infty(f)$.

The main Theorem in this section is that the closure of $\cM_n(X)$ is in the transalgebraic class. More precisely,

\begin{theorem}\label{thm:main}
The closure of $\cM_n(X)$ in  $\cF(X)$ is 
$$
\overline{\cM_n(X)}\subset \bigcup_{\substack{d_0, d_\infty \geq 0 \\ d_0+d_ \infty\leq n}} \cT^{d_0, d_\infty}(X)
\subset \cM_n(X)\cup\cT_{n-1}(X)\subset \cT_{n}(X) \subset \cT(X)
$$ 

For $n\leq |S|$, the closure of $\cM_n(X,S)$ is 
$$
\overline{\cM_n(X, S)}\subset \bigcup_{\substack{d_0, d_\infty \geq 0 \\ d_0+d_ \infty\leq n}} \cT^{d_0, d_\infty}(X,S)
\subset \cM_n(X,S)\cup\cT_{n-1}(X,S)\subset \cT_{n}(X,S)\subset \cT(X,S)
$$ 
\end{theorem}

The following two Lemmas are clear from the local analysis.

\begin{lemma}\label{lemma1}
Let $f\in \cM(X)$ be a meromorphic function, $f:X\to \overline{\CC}$. The poles and
zeros of $f$ correspond bijectively to simple poles of the
logarithmic derivative form $f'/f\, dz=d \log f$. The residue at these simple poles is the 
positive, resp. negative, multiplicity of the zero, resp.  pole, of $f$. 
% Therefore, 
% the degree of the logarithmic derivative $d \log f$ is bounded by the number 
% of zeros and poles \underline{not} counted with multiplicity.
\end{lemma}

% Another Lemma that follows from the previous local analysis is,

\begin{lemma}\label{lemma2}
Let $f\in \cT(X)$. The logarithmic derivative  $d \log f$ is a meromorphic differential $d \log f\in \Omega^1(X)$ with integer residues.
Conversely, if $f\in \cF(X)$ is such that $d \log f\in \cM\Omega^1(X)$ has integer residues at poles, then $f\in \cT(X)$. 
\end{lemma}

The following Lemma shows that the only way to have a creation of an exponential singularity $\rho$ of 
transalgebraic degree $d_\rho(f)$ at the limit 
for a converging sequence $f_k\to f$ of rational functions is to have $d_{\rho}(f)+1$ distinct 
sequences of poles and zeros of the $f_k$ converging to $\rho$ (there must be both, poles and zeros).

\begin{lemma}
Let  $(f_k)\subset  \cM_n(X)$ converging to $f\in \cF(X)$ and let $\rho \in S(f)\subset X$. Then there 
exists at least $d_{\rho}(f)+1$ distinct sequences of poles and zeros of the $(f_k)$ converging to $\rho$. 
\end{lemma}

\begin{proof}
Consider a local chart at $\rho$ and the logarithmic derivative $f_k'/f_k$ in this chart that has simple poles 
corresponding to zeros and poles of $f_k$. We have $f_k'/f_k\to f'/f$ and at the limit we have 
at $\rho$ a pole of order $d_{\rho}(f)+1$ for $f'/f$. The result follows from Rouch\'e's Theorem. Note that 
poles and zeros can also annihilate each other is the have the same multiplicity.
\end{proof}

\begin{proof}[Proof of Theorem \ref{thm:main}]
We consider a sequence of meromorphic functions $(f_k)\subset  \cM_n(X)$ converging to $f\in \cF(X)$. The 
associated sequence of meromorphic logarithmic derivatives $d \log f_k$ have 
the support of their divisor Hausdorff converging to $\supp (\Div(f))$, and uniformly on compact sets outside $X-\supp(f)$ we have 
$d \log f_k \to d \log f$. If $m$ poles of the $d \log f_k$ converge to a pole of $d \log f$ then the order of the pole is less 
or equal to $m-1$ and the residue is an integer as the sum of integer residues of $d \log f_k$ of the converging points.
Hence using the second Lemma \ref{lemma2} and counting poles 
we have that $f\in \cT^{d_0,d_\infty}(X)$.
We have the same  proof for the closure of $\cM_n(X,S)$.
\end{proof}

Observe that from the proof we have that if the residues of the poles of the 
$d \log f_k$ (or the order of the zeros or poles of $f_k$) that collapse into a pole of $d \log f$ are bounded, 
then the pole of $d \log f$ must be simple. Hence, the only way to have higher 
order poles for $d \log f$ is when poles and zeros of $f_k$ collapse into a 
point, the orders are unbounded, but their sum is asymptotically constant. 

We have a converse, and any $f\in \cT(x)$ can be approximated by a sequence $(f_k)\subset \cM(X)$, and each 
exponential singularity can be realized as a limit of poles and zeros, but these more precise results will 
be studied elsewhere, since we are interested in this article on the simpler case $X=\overline{\CC}$. We 
specialize the above results to $X=\overline{\CC}$.

\medskip

Observe that  $\cM_n=\cM_n(\overline {\CC})$ (we drop the dependence on $X$ since $X=\overline{\CC}$ from 
now on) is the group  of non-zero rational functions
$R: {\overline {\CC}} \to {\overline {\CC}}$
with support of cardinal bounded by $n$, i.e. such that
$$
|{\hbox {\rm supp}} (R)|=|R^{-1}(0) \cup R^{-1}(\infty )| \leq n \ .
$$
We abuse the notation by writting $R^{-1}(0)$, resp. $R^{-1}(\infty)$ to 
denote the set of zeros, resp. poles, of $R$, including the possible one at $\infty \in \overline {\CC}$.
Note that we do not count multiplicities. 
So we have
$$
\cM_n =\CC (z) \cap \cF_n \subset \cF_n\ .
$$

\begin{theorem}\label{thm:transalgebraic_Riemann_sphere}
The group of transalgebraic functions $\cT_n$ are the functions 
$f\in \cF_n$ of the form
$$
f=R_0e^{R_1}
$$
where $R_0\not=0$ and $R_1$ are meromorphic functions with $R_1=0$ or

\begin{equation}\label{inequality}
|R_0^{-1}(0) \cup R_0^{-1}(\infty )| +\deg R_1 \leq n \ . 
\end{equation}
In particular, the group of transalgebraic functions $\cT$ is
$$
\cT =\{ f=R_0e^{R_1}; R_0, R_1 \in \CC(z); R_0\not=0 \}
$$
\end{theorem}

\begin{proof}For $n=0$ the result is clear, so we assume $n\geq 1$.
Such a function $f=R_0e^{R_1}$ is clearly in $\cT_n$. Conversely, given a 
function $f\in \cT_n$, then the degree of the divisor $\Div_0(f)$ is $0$ and 
we can choose a rational function $K_0$ such that 
$$
\Div K_0 = \Div_0(f) \ .
$$
We can also choose a rational function $K_1$ with polar part matching the polar 
part of $d \log (f/K_0)$, in particular with the same integral residues.
We consider the primitive 
$$
\exp \left (\int K_1 \right ) 
$$
which is of the form
$$
\exp \left (\int K_1 \right ) =L_0 \exp (R_1)
$$
where $L_0$ is a rational function coming from the integration of the 
order $1$ polar part of $K_1$, and $R_1$ from the higher order polar part and 
both are rational functions since $\int K_1$ has no monodromy at the support of $\Div_\infty (f)$, and 
$$
R_1^{-1}(\infty)=\Div_\infty (f) \ .
$$
Now, for $R_0= L_0 K_0$ we have that $R_0e^{R_1}/f$ is a meromorphic function with no zeros nor poles, thus it is a 
constant and we can multiply 
$R_0$ by a non-zero constant so that $f=R_0e^{R_1}$.
\end{proof}

It is instructive to understand how there transalgebraic functions arise as a limit of rational functions.
When $f=R_0e^{R_1}$  we have when $k\to +\infty$,
$$
f_k=R_0 \left (1+\frac{R_1}{k} \right )^k \to f 
$$
and $f_k\in \cM_n(\overline {\CC})$ because of the inequality (\ref{inequality}).

Conversely, let $f_k \to f \in \cF_n(\overline {\CC})$ with $f_k \in \cM_n(\overline {\CC})$. The zeros and poles 
collapse into the divisor of $f$. Only when zeros and poles cohalesce with the sum of residues becoming asymptotically 
constant that we can have the emergence of an exponential singularity of finite order for $f$.

We note a particular case of the previous theorem in the next Corollary.

\begin{corollary}\label{cor:polynomials}
For a positive integer $n\geq 0$, consider the
space of non-zero polynomials $\cP_n \subset \CC[z]^*$ 
normalized having exactly $n$ zeros, not counted with
multiplicity. We endow this space with the topology of uniform
convergence on compact sets off the zeros as before.
Then we have
$$
{\overline {\cP}_n}=\cP_n \cup \cT\cP_n \ ,
$$
where $\cT\cP_n$ is the space of functions of the form $f=P_0e^{P_1}$ where
$P_0 \in \CC[z]^*$ is non-zero and $P_1\in \CC (z)$, with
$$
|P_0^{-1}(0)| +\deg P_1 +1\leq n \ .
$$
\end{corollary}

\begin{proof}

From the previous Theorem we get that all limits are of the form $f=R_0e^{R_1}$ but the limit 
is holomorphic on $\CC$ hence $R_0=P_0$ and $R_1=P_1$ are polynomials, and $f=P_0e^{P_1}$ and
$P_0$ is not identically $0$.
Each polynomial in $P\in \cP_n$ is a rational function with
$$
|P^{-1}(0) \cup P^{-1}(\infty )|= |P^{-1}(0)| \cup \{\infty \} =n+1 \ .
$$
We apply the general theorem getting
$$
|P_0^{-1}(0) \cup P_0^{-1}(\infty )|+\deg P_1 +1= |P_0^{-1}(0)|+\deg P_1 +2\leq n+1 
$$
and the result follows,
$$
|P_0^{-1}(0)|+\deg P_1 +1\leq n 
$$
\end{proof}

\bigskip

\textbf{A particular case.}

\medskip

Let $(P_k)$ be a sequence of polynomials with exactly $n$ zeros, all escaping
to $\infty$. If they are normalized conveniently in order to have a limit (for
example, such that $P_k(0)=1$, $P_k'(0)=1$), then the limit has no finite zero, nor pole, thus the limit must
be of the form $\exp (P_1)$ where $P_1$ is a polynomial of degree at most $n$.

\medskip

\textbf{Historical comments.}

\medskip

 Fields generated by functions with exponential singularities of finite order on a 
compact Riemann surfaces $X$ of genus $g>0$ have been studied by  
P. Cutillas Ripoll  in a series of remarkable papers starting 
with \cite{Cu1} (see also \cite{Cu2} and \cite{Cu3} and more recent articles). Cutillas proves the existence 
of minimal field of functions associated to the compact Riemann surface containing the space 
of meromorphic functions and realizing any divisor on $X-S$, $S$ finite and non-empty. 
These minimal fields are all isomorphic and  independent of $S$, so there is an abstract Cutillas field $\cC(X)$ associated to any compact 
Riemann surface $X$. 
Moreover, he generates these function fields using functions in $\cT(X,S)$, i.e. 
with exponential singularities located at $S$. This fantastic result seems to not be well 
known\footnote{Cutillas article \cite{Cu1} from 1984 has no citations according to Math Reviews, which shows how misguided
is modern research.}.
In \cite{Cu1} it is pointed out that these functions have been considered by Clebsch and Gordan, and also by Weierstrass.
They appear in chapter VII of the treatise by Baker  \cite{Ba}  where in 
footnote remarks there is some not very precise reference to 
Weierstrass work (see also \cite{Ba1907}). Continuing the work of Baker, that was somewhat forgotten, 
functions with exponential singularities on hyper-elliptic curves have been used, under the name of Baker-Akhiezer functions, 
by the russian school to construct explicit solutions of KP and KdV equations, see the surveys \cite {Du} and \cite{BEL2012}. 

Transalgebraic functions on the Riemann sphere were studied by Nevanlinna (\cite{Ne1932}, \cite{Ne1953}), 
and Taniguchi (\cite{Ta2001}, \cite{Ta2002}). They 
also appear naturally as uniformizations of log-Riemann surfaces defined
by K. Biswas and the author (\cite{BPM1}, \cite{BPM2}, \cite{BPM3}) which roughly speaking are Riemann surfaces with 
canonical flat charts. In particular, in \cite{BPM2}  
the Caratheodory convergence  of log-Riemann surfaces is defined, and a generalization to this setting of 
Caratheodory's Kernel Convergence Theorem is proved (Theorems 1.1 and 1.2 in \cite{BPM2}).  As Corollary (Corollary 
1.3 in \cite{BPM2})) a purely geometric proof of Euler's limit (which is central to this article) is obtained,
$$
e^{z}=\lim_{n\to +\infty} \left (1+\frac{z}{n}\right )^n
$$
Transalgebraic curves are defined in \cite{BPM3}. They generalize classical algebraic curves allowing 
infinite ramification points. More precisely, they are defined as log-Riemann surfaces having a finite number of finite (algebraic) and 
infinite (transcendental) ramification points. The Caratheodory closure of algebraic curves 
with uniformly bounded number of ramification points are proved to be transalgebraic curves (Theorem 2.11 in 
\cite{BPM3}). This result is the geometric counterpart of the previous Corollary \ref{cor:polynomials}. The 
more general closeness Theorem \ref{thm:main} is related to uniformizations of higher 
genus log-Riemann surfaces and more precisely to the main 
Theorem in \cite{BPM4}. 
More precisely, it is proved in \cite{BPM4} that any log-Riemann surface of finite topology (finitely generated fundamental group), 
is biholomorphic to a pointed compact Riemann surface $X-S$ equipped with a transalgebraic differential form $\omega\in \cT\Omega^1(X)$, i.e. 
a differential form that is locally of the form $\omega=f(z) dz$ with $f$ with exponential singularities, holomorphic out of $S$.  
Some of the transalgebraic properties of periods of transalgebraic curves are discussed in \cite{BPM5}.

\medskip

%%%%%%%%%%%%%%%%%%%%%%%%%%%%%%%%%%%%%%%%%%%%%%%%%%%%%%%%%%%%%%%%%%
\section{E\~ne product structure on the transalgebraic class of $\PP^1(\CC)$.} \label{sec:infi}
%%%%%%%%%%%%%%%%%%%%%%%%%%%%%%%%%%%%%%%%%%%%%%%%%%%%%%%%%%%%%%%%%%

In this section we extend the e\~ne product to the transalgebraic class of the Riemann 
sphere $\cT(\overline{\CC})$. The starting observation is the remarkable 
Convolution formula from \cite{PM1} (Theorem 5.3)
$$
e^{\frac{z}{1-z}}\star f(z) =\exp \left (\sum_\alpha
\frac{z}{\a-z}\right )=\prod_\alpha e^{
\frac{z}{\a-z}}
$$
where $(\alpha)$ is the sequence of zeros of $f$.
Observe that $f(z)=e^{\frac{z}{1-z}} \in \cT^{0,1}\subset \cT_1$ is a non-meromorphic transalgebraic function
with a single point support for its divisor
$$
\Div (f)= (1)_\infty
$$
The function $f(z)=e^{\frac{z}{1-z}}$ is just the exponential function pre-composed by the Moebius transform
mapping $\infty$ to $z=1$ and tangent to the identity at $0$. The convolution formula 
is what to expect if we consider the exponential 
singularity at $1$ of $e^{\frac{z}{1-z}}$ as a zero of infinite order.

We define a sequence of rational functions $(R_k)_{\geq 1}$ appearing in Euler's computations in
\cite{Eu1768} p.85 (it is indeed $-z^{-1}R_k(-z)$ that Euler considers) for the summation of integer powers (see also \cite{Ca2002} 
where the rational functions  $\Phi_k$ are considered, and $R_k =-\Phi_{k-1}$),
$$
R_1(z)=\frac{z}{z-1}=-\sum_{n=1}^{+\infty} z^n
$$
and for $k\geq 0$,
$$
R_{k} = R_1\star_e \ldots \star_e R_1=R_1^{\star_e k} 
$$
Therefore, for $k\geq 1$,
$$
R_{k} (z) =-\sum_{n=1}^{+\infty} n^{k-1} z^n
$$
and we can also define these rational functions by $R_k(0)=0$ and 
\begin{equation}\label{eq:recurrence}
R_{k+1} =z \frac{d R_{k}}{dz}  
\end{equation}
There follow the first seven rational function listed by Euler in \cite{Eu1768}
\begin{align*}
R_1(z) &=-\frac{z}{1-z} \\
R_2(z) &=-\frac{z}{(1-z)^2} \\
R_3(z) &=-\frac{z(1+z)}{(1-z)^3}\\
R_4(z) &=-\frac{z(1+4z+z^2)}{(1-z)^4} \\
R_5(z) &=-\frac{z(1+11z+11z^2+z^3)}{(1-z)^5} \\ 
R_6(z) &=-\frac{z(1+26z+66z^2+26z^3+z^4)}{(1-z)^6} \\
R_7(z) &=-\frac{z(1+57z+302z^2+302z^3+57z^4+z^5)}{(1-z)^7}
\end{align*}
In general we have,
\begin{proposition}\label{prop:Euler_rat}
For $k\geq 0$, we have
\begin{equation}\label{eq:rec_rat}
R_k(z)=-\frac{zP_k(z)}{(1-z)^k} 
\end{equation}
where $P_k \in \ZZ[z]$ and for $k\geq 1$,
\begin{equation}\label{eq:rec_pol}
P_{k+1}(z)=(1+(k-1)z)P_k(z)+z(1-z)P'_k(z)
\end{equation}
and $P_1=1$. We have for $k\geq 2$, $\deg P_k =k-2$, $P_k(1)=(k-1)!$, $P_k(0)=1$, and the functional 
equations,
\begin{align*}
R_k(z^{-1}) &=(-1)^k R_k(z) \\
P_k(z^{-1}) &=z^{2-k}P_k(z)  
\end{align*}
Hence, for $k\geq 2$, $R_k$ vanishes at $0$ and $\infty$, and 
has only one pole of order exactly $k$ at $1$.
\end{proposition}

\begin{proof}
By induction we get the recurrence (\ref{eq:rec_rat}), and the polynomial 
recurrence (\ref{eq:rec_pol}) follows from it. Making $z=1$ in (\ref{eq:rec_pol})
we have for $k\geq 1$, $P_{k+1}(1)=k P_k(1)$ and $P_1(1)=1$, hence $P_k(1)=(k-1)!$.
Making $z=0$ in (\ref{eq:rec_pol}) we have for $k\geq 1$, $P_{k+1}(0)=P_k(0)=P_1(0)=1$.
The two  functional equations are equivalent. If we define $Q_k(z)=(-1)^{k+1} R_k(z^{-1})$, we check
$$
Q_{k+1}(z)=zQ_k'(z) 
$$
and $Q_k(0)=0$, $Q_2=R_2$, thus $Q_k=R_k$.

Now, it is clear that for $k\geq 2$, $R_k$ vanishes at $0$ and $\infty$, and 
has only one pole of order exactly $k$ at $1$.
\end{proof}

Following Euler's intuition described in section \ref{sec:eule}, and the results from the previous section,
it is natural to define a zero of infinite order \textit{at a finite place} $z_0\in \CC$ "\`a la Euler":

\begin{definition} We define symbolically for $z_0\in \CC^*$ 
$$
\left (1-\frac{z}{ z_0}\right )^\infty \equiv \exp \left (\frac{z}{z_0-z} \right ) =e^{R_1(z/z_0)}
 \ .
$$
\end{definition}

With this notation, the convolution formula can be rewritten as,
\begin{equation*}
\left (1-\frac{z}{z_0}\right )^\infty \star f(z) =\prod_\alpha \left (1-\frac{z}{z_0 \alpha}\right )^\infty 
\end{equation*}
which is just distributivity with respect  to infinite products and transalgebraic divisors
\begin{align*}
\left (1-\frac{z}{z_0}\right )^\infty \star f(z) &= \left (1-\frac{z}{z_0}\right )^\infty \star 
\prod_\alpha \left ( 1-\frac{z}{\alpha }\right ) \\
&= \prod_\alpha \left (1-\frac{z}{z_0}\right )^\infty \star \left ( 1-\frac{z}{\alpha }\right ) \\
&=\prod_\alpha \left (1-\frac{z}{z_0 \alpha}\right )^\infty \ .
\end{align*}

Now we define higher order zeros,

\begin{definition} For $z_0\in \CC^*$, we define symbolically,
$$
\left (1-\frac{z}{z_0}\right )^{k.\infty}  \equiv  e^{R_k(z/z_0^k)} \in \cT^{0,k}\cap \cT_1 \subset \cT
$$
\end{definition}

Note that $f_k(z)=\left (1-\frac{z}{z_0}\right )^{k.\infty}$ is the function of exponential type 
with a single point transcendental divisor of order $k$  at $z^k_0\in \CC^*$,
$$
\Div (f_k)=\Div_\infty (f_k) = k.(z^k_0)_\infty
$$
Note the new fact that taking the e\~ne power of a simple infinite order zero at $z_0$, changes the support of the 
new infinite zero of order $k$ to $z_0^k$.
From the exponential form of the e\~ne product we get

\begin{proposition}
We have for $k\geq 0$ and $z_0\in \CC^*$,
 $$
\left ( e^{\frac{z}{z_0-z} }\right )^{\star k}= e^{\frac{z}{z_0-z}} \star \ldots
\star e^{\frac{z}{z_0-z}} =e^{R_k(z/z_0)} \ ,
$$
in particular
$$
\left ( e^{\frac{z}{1-z}} \right )^{\star 0}=1-z \ .
$$
and this can be written \`a la Euler
$$
\left (1-\frac{z}{z_0}\right )^{k.\infty} =
\left (\left (1-\frac{z}{z^k_0}\right )^\infty \right )^{\star k}
= e^{R_k(z/z_0^k)}
$$
More generally, for $n\geq 1$ we have, 
$$
e^{ R_{k_1}(z/z_1)}  \star \ldots \star  e^{ R_{k_n}(z/z_n)} = 
e^{ R_{k_1+\ldots +k_n} (z/(z_1\ldots z_n))} \ .
$$
or, \`a la Euler,
$$
\left (1-\frac{z}{z_1}\right )^{k_1.\infty} \star \ldots \star \left (1-\frac{z}{z_n}\right )^{k_n.\infty}= 
\left (1-\frac{z}{z_1\ldots z_n}\right )^{(k_1+\ldots +k_n).\infty} \ .
$$
\end{proposition}

The proof is clear from the definitions. It is satisfactory  to check that this is the expected 
result from Euler heuristics. For example, we have the following formal computation \`a la Euler,
for $z_1, z_2 \in \CC^*$,

\begin{align*}
\left (1-\frac{z}{z_1}\right )^\infty \star
\left (1-\frac{z}{z_2}\right )^\infty &=
\prod_\infty \left (1-\frac{z}{z_1}\right ) \ \star \
\prod_\infty \left (1-\frac{z}{z_2}\right ) \\
&= \prod_{\infty , \infty} \left (1-\frac{z}{z_1}\right )
\star \left (1-\frac{z}{z_2}\right ) \\
&=\prod_{\infty , \infty} \left (1-\frac{z}{z_1 z_2}\right ) \\
&= \left (1-\frac{z}{z_1 z_2}\right )^{\infty . \infty} \\
&= \left (1-\frac{z}{z_1 z_2}\right )^{2. \infty } 
\end{align*}

Now, it easy to see that these transalgebraic functions generate the multiplicative 
group of transalgebraic functions on the Riemann sphere:

\begin{theorem}\label{thm:transgenerators}
The multiplicative group $(\cT(\overline{\CC}), .)$ is generated by the non-zero meromorphic functions 
$\cM(\overline{\CC})^*$, $e^{P(z)}$, $e^{P(1/z)}$, with $P\in \CC[z]$, and for $z_0\in \CC^*$, $\alpha \in \CC$, $k\geq 1$,
$$
e^{\alpha R_k(z/z_0)}=\left (\left (1-\frac{z}{z_0}\right )^{k.\infty}\right )^\alpha =\left (1-\frac{z}{z_0}\right )^{\alpha .(k.\infty )}
$$
\end{theorem}

\begin{proof}

Using Theorem \ref{thm:transalgebraic_Riemann_sphere} any function $f\in \cT(\overline{\CC})$ is of the form
$f=R_0e^{R_1}$ with $R_0, R_1 \in \CC$, $R_0\not= 0$. 
% Multiplying by functions of the form $e^{P(z)}$ and $e^{P(1/z)}$ we can assume 
% that $0$ and $\infty$ are regular points of $R_1$, i.e. no poles at $0$ or $\infty$. Then 
By Proposition \ref{prop:Euler_rat}
the rational functions $R_k(z/z_0)$ can be used to reconstruct any polar part in $\CC^*$, so we can find a finite linear 
combination of functions $R_k(z/z_0)$ such that 
$$
R_1-\alpha_0 P_0(1/z) -\alpha_\infty P_\infty(z)-\sum_{z_0}\sum_{k=1}^{k(z_0)} \alpha_{k,z_0} R_k(z/z_0)
$$
has no poles in the Riemann sphere, hence it is a constant. 
Therefore, we have that 
$$
f.\prod_{k, z_0} e^{-\alpha_{k,z_0} R_k(z/z_0)} e^{-\alpha_0 P_0(1/z)} e^{-\alpha_\infty P_\infty(z)}  \in \CC(z)^*
$$
is a non-zero rational function and the result follows.
\end{proof}

\medskip
\textbf{Extension of the e\~ne product to $\cT$.}
\medskip

In view of the previous factorization given by Theorem \ref{thm:transgenerators}, and the definition of the e\~ne product from 
\cite{PM1}, we already have the e\~ne product in the subgroup of $\cT$ of elements without a pole or singularity at $0$. 
Using the projective invariance from Theorem 11.4 from \cite{PM1} for rational functions,
$$
f(1/z)\star g(1/z) = f\star g (1/z)
$$
we can extend the exponential form  of the e\~ne product to Laurent developments in the exponential at $0$, i.e. if
\begin{align*}
f(z) &=e^{F(z)} \\ 
g(z) &=e^{G(z)} 
\end{align*}
with 
\begin{align*}
F(z) &=\sum_{k\in \ZZ} F_k z^k \\ 
G(z) &=\sum_{k\in \ZZ} G_k z^k
\end{align*}
then we have
$$
f\star g(z) = e^{H(z)}
$$
with
$$
H(z)=- \sum_{k\in \ZZ} kF_k G_k z^k = F\star_e G (z)
$$
where we denote $\star_e$ the extension of the linearized exponential form of the e\~ne product
$$
F\star_e G (z)=-\sum_{k\in \ZZ} kF_k G_k \, z^k
$$
With this definition we extend the e\~ne product to the full group $\cT/\CC^* =\bar \cT$, i.e. $\cT$ modulo non-zero constants.
Note that the e\~ne product with the constant $1$ is the constant $1$ that plays the role of additive zero in the following ring structure,

\begin{theorem}
$( \bar \cT, . ,\star)$ is a commutative ring.
\end{theorem}

The e\~ne product is clearly continuous on the generators of the transalgebraic class for the 
topology defined in Section \ref{sec:eule}. We get

\begin{theorem}
$( \bar \cT, . ,\star)$ is a topological commutative ring.
\end{theorem}

Also it is clear that we have for $n,m\geq 0$,
$$
\bar \cT_n \star \bar \cT_m \subset \bar\cT_{nm}
$$
hence we have 
\begin{theorem}
$( \bar \cT, . ,\star)$ is a graded topological commutative ring.
\end{theorem}

\medskip
\textbf{Remark on further algebraization.}
\medskip

We can observe that if we consider a subfield $\QQ\subset K\subset \CC$ (for example a number field $K$), we can define $\cT(\PP^1(K))$ as the sub-group of $\cT$ with 
functions with exponential singularities at places in $\PP^1(K)\subset \PP^1(\CC)$ such that $d \log f \in \Omega^1_{/K}(\PP^1(K))$.
Then, since the Euler rational functions have rational coeffidints, $R_k(z)\in \QQ(z)$, we can check that the proof of Theorem \ref{thm:transgenerators} goes through
(the coefficients $(\alpha_{k,z_0})$ are elements of $K$), and the e\~ne product extends to $\bar \cT(\PP^1(K))=\cT(\PP^1(K))/K^*$ and 
$$
(\bar \cT(\PP^1(K)), ., \star)
$$
is a graded topological commutative ring.

For an algebraic curve $X$ defined over a field $K$ we can also define $\cT(X,K)$, the transalgebraic class over $K$, and so on. 

One of the magic in Euler's computations is its symbiotic analytic-algebraic content. Thus, it is not surprising 
that the e\~ne structures allow such algebraizations. We leave these rich  algebraic extensions for future articles.

\medskip

\textbf{E\~ne poles.}

\medskip

Once we have defined zeros of infinite order at finite places,
it is natural to ask for the definition of poles of infinite
order at finite places. 
Polylogarithm functions appear then
naturally. Recall that polylogarithms are defined for $k\geq 1$ by
$$
\Li_k(z)=\sum_{n=1}^{+\infty} n^{-k} z^n  \ .
$$
We have $\Li_k(0)=0$, the series has radius of convergence $1$, and we have 
a singularity at $1$ which is a branching point. On the other sheets of its log-Riemann surface we have 
branchings at $0$ and $1$ for $k\geq 2$, and only at $1$ for $k=1$ since 
$\Li_1(z)=-\log (1-z)$ (see for example \cite{Oesterle1993} for these geometric information 
and more properties of polylogarithms).

We can complete the sequence $(R_k)_{k\geq 1}$ of Euler rational function to indexes 
$k<0$ using the differential recurrence \ref{eq:recurrence} and the condition $R_k(0)=0$, and we get
$$
R_0(z)=\Li_1(z)=-\log(1-z)
$$
and for $k\leq 0$,
$$
R_k(z)=\Li_k(z)
$$
\begin{proposition}
We have
$$
e^{R_k(z)} \star e^{-Li_{k+1}(z)} =1-z \ .
$$
\end{proposition}

\begin{proof}
This follows directly from the exponential form, for $k\geq 0$,
$$
R_k\star_e (- \Li_{k+1})=\log (1-z) \ .
$$
\end{proof}
So we have
$$
\left (1-\frac{z}{1}\right )^{k. \infty} \star e^{- \Li_{k+1}(z)} =1-z .
$$
and it is natural to define the e\~ne-pole of infinite order $k\geq 1$ at $1$
as
$$
\left (1-\frac{z}{1}\right )^{-k. \infty } = e^{- \Li_{k+1}(z)}
$$
More generally, we define the pole of infinite order
$k$ ``at'' a finite place $z_0\not= 0, \infty$ as
\begin{definition} For $z_0\in \CC^*$, we symbolically define
$$
\left (1-\frac{z}{z_0}\right )^{-k. \infty } = e^{- \Li_{k+1}(z z_0^k)} \ ,
$$
as the function with divisor a single pole of infite order $k$ at $z_0\in \CC^*$ (but the singularity is at $z_0^{-k}$).
\end{definition}

Note that this time these functions  don't have  exponential singularities, but  branching 
singularities with non-trivial monodromy.  The function with a single
$k$-infinite pole at $z_0$ has a branching point located at $z_0^{-k}$. 
These functions take the 
value $1$ at $0$ (in their principal determination) and the e\~ne product is well defined through the 
exponential form. The important observation is that these type of singularities do 
appear naturally. Thus, it is natural to extend the e\~ne product to functions with singularities 
with non-trivial monodromy. This will be treated in subsequent articles.

With these definition we conclude, always \`a la Euler,

\begin{theorem}
For $k,l \in \ZZ$, $z_1, z_2 \in \CC$, $z_1, z_2\not= 0,\infty$,
we have
$$
\left (1-\frac{z}{z_1} \right )^{k.\infty } \star
\left (1-\frac{z}{z_2} \right )^{l.\infty }
=\left (1-\frac{z}{z_1^k z_2^l} \right )^{(k+l).\infty }  \ \ .
$$
\end{theorem}

\end{document}